\documentclass[11pt,a4paper]{article}
\usepackage[utf8]{inputenc}
\usepackage[T1]{fontenc}
\usepackage[margin=1in]{geometry}
\usepackage{amsmath,amssymb,amsthm,mathtools}
\usepackage{enumitem}
\usepackage[hidelinks]{hyperref}
\usepackage{booktabs}
\usepackage{microtype}
\usepackage{longtable}
\usepackage{array}
\usepackage{xurl}

\newtheorem{theorem}{Theorem}
\newtheorem{lemma}[theorem]{Lemma}

\newtheorem{conjecture}[theorem]{Conjecture}
\newtheorem{proposition}[theorem]{Proposition}
\newtheorem{corollary}[theorem]{Corollary}
\theoremstyle{definition}

\title{An independent computer-assisted proof of the Chen--Raspaud conjecture for \texorpdfstring{$k=4$}{k=4}}
\author{Marysia Nazarczuk\\
Faculty of Mathematics, Informatics and Mechanics\\
University of Warsaw, Poland}
\date{September 2026}

\begin{document}
\maketitle

\begin{abstract}
We give an independent computer-assisted proof of the $k=4$ case of the
Chen--Raspaud conjecture. We prove that every graph $G$ with
$\operatorname{odd\mbox{-}girth}(G)\ge9$ and
$\operatorname{mad}(G)<9/4$ admits a homomorphism to the Kneser graph
$K(9,4)$. The proof combines a minimal-counterexample argument with a
rooted star replacement, exact finite computations in $K(9,4)$, and a final
charging argument. After all reducible local types are removed, the unique
positive local type is $(3,3,4)$. Its unit excess is transferred through
its $4$-thread to a relative with sufficient negative capacity. The
computer-assisted statements used in the proof are certified by exact C++
bitset verifiers, and a separate Python implementation provides an independent
cross-check; the complete source code and raw certificates accompany the
manuscript.
\end{abstract}



\section{Introduction}

Graph homomorphisms provide a common language for many variants of graph
coloring; see, for example, Hell and Ne\v{s}et\v{r}il~\cite{HellNesetril2004}.
For integers $a\ge 2b$, the Kneser graph $K(a,b)$ has vertex set
$\binom{[a]}b$, and two vertices are adjacent precisely when the
corresponding $b$-subsets are disjoint. Thus a homomorphism
$G\to K(a,b)$ is exactly an $(a:b)$-coloring of $G$: every vertex receives
$b$ colors from a palette of size $a$, and adjacent vertices receive
disjoint color sets. Multiple colorings in this sense go back at least to
Stahl~\cite{Stahl1976}; they form one of the basic viewpoints on fractional
coloring~\cite{ScheinermanUllman1997}.

Kneser graphs are central objects in algebraic and topological graph theory.
Lov\'asz's proof of Kneser's conjecture determined their chromatic number
\cite{Lovasz1978}. The subfamily
\[
   K(2k+1,k),\qquad k\ge1,
\]
is usually called the family of \emph{odd graphs}; its structure has been
studied from several points of view, including work of Meredith and
Lloyd~\cite{MeredithLloyd1973} and Biggs~\cite{Biggs1979}. The Petersen
graph is $K(5,2)$. Poljak and Tuza determined the odd girth of general
Kneser graphs~\cite{PoljakTuza1987}; in particular,
\[
   \operatorname{odd\mbox{-}girth}(K(2k+1,k))=2k+1.
\]
We include a short self-contained proof of the instance
$\operatorname{odd\mbox{-}girth}(K(9,4))=9$ below.

There is a natural interaction between sparsity and homomorphisms. Recall
that
\[
 \operatorname{mad}(G)=\max_{H\subseteq G}\frac{2|E(H)|}{|V(H)|}
\]
is the maximum average degree. The classical degeneracy argument, already
implicit in Szekeres and Wilf~\cite{SzekeresWilf1968}, shows that a strict
upper bound on $\operatorname{mad}(G)$ yields an ordinary coloring bound.
For homomorphisms to odd graphs, however, sparsity alone cannot suffice:
if $G\to H$, then every odd cycle of $G$ maps to an odd closed walk in $H$,
so
\[
 \operatorname{odd\mbox{-}girth}(G)\ge
 \operatorname{odd\mbox{-}girth}(H).
\]
This makes an odd-girth assumption unavoidable when $H=K(2k+1,k)$.

The same circle of ideas appears in results on planar and structurally
restricted graph classes. Klostermeyer and Zhang proved that planar graphs
of sufficiently large odd girth admit $(2k+1:k)$-colorings
\cite{KlostermeyerZhang2002}; see also the fractional-coloring perspective
of Pirnazar and Ullman~\cite{PirnazarUllman2002}. For the Petersen target,
Dvo\v{r}\'ak, \v{S}krekovski, and Valla proved that every planar graph of
odd girth at least $9$ maps to the Petersen graph
\cite{DvorakSkrekovskiValla2008}. Homomorphism bounds for
$K_4$-minor-free graphs provide another setting in which Kneser targets and
odd girth interact particularly cleanly~\cite{BeaudouFoucaudNaserasr2017}.
These results concern special graph classes; the conjecture below instead
asks for a condition expressed only through maximum average degree and odd
girth.

Chen and Raspaud~\cite{ChenRaspaud2010} formulated the following conjecture.

\begin{conjecture}[Chen--Raspaud]\label{conj:CR}
Let $k\ge 2$ and let $G$ be a graph such that
\[
\operatorname{odd\mbox{-}girth}(G)\ge 2k+1
\qquad\text{and}\qquad
\operatorname{mad}(G)<\frac{2k+1}{k}.
\]
Then $G$ admits a homomorphism to $K(2k+1,k)$.
\end{conjecture}

Chen and Raspaud proved the case $k=2$~\cite{ChenRaspaud2010}. The case
$k=3$ was proved by \L{}yczek, Nazarczuk, and Rz\k{a}\.{z}ewski
\cite{LyczekNazarczukRzazewski2024}, using a minimal counterexample,
reducible configurations, and discharging. The present paper studies the
next case.

\begin{theorem}\label{thm:main}
Let $G$ be a graph such that
\[
\operatorname{odd\mbox{-}girth}(G)\ge 9
\qquad\text{and}\qquad
\operatorname{mad}(G)<\frac94.
\]
Then $G$ admits a homomorphism to $K(9,4)$.
\end{theorem}

\paragraph{Relation to independent work.}
The mathematical research leading to the present proof was carried out through
the end of 2025, and the proof was written up in April 2026. While this
manuscript was being prepared for public release, Choi independently published
a proof of the same $k=4$ case~\cite{Choi2026}. The two arguments share the
broad minimal-counterexample, thread-reduction, and discharging framework, but
their local mechanisms are different. Choi's proof is computer-free and
develops a thread-compatibility and replacement approach that also yields
computer-free proofs for $k=2$ and $k=3$. The present proof instead works
directly with the finite geometry of $K(9,4)$, uses the rooted-star replacement
and exact exhaustive extension checks, and closes by isolating the exceptional
type $(3,3,4)$ and transferring its single unit of excess. The present proof was developed independently and does not rely on any
result or argument from~\cite{Choi2026}.

\paragraph{What is new in the proof.}
The difficulty in passing from $k=3$ to $k=4$ is not the size of the target
alone: $K(9,4)$ has $126$ vertices, but the number of possible local
thread-length patterns around a vertex grows rapidly, and naive pairwise
bookkeeping does not close the final discharging inequalities. The proof
below packages the local extension problem into two exact mechanisms. The
first is a \emph{rooted star replacement}: after one incident thread is
chosen as a root, the center of a star of threads is removed and the
remaining relatives are connected to the root relative by paths whose
lengths preserve the relevant parity and sparsity constraints. The second
is a finite extension test in $K(9,4)$, implemented as an exact set-cover
problem. Together with a simple forbidden-set union bound, these tools
eliminate every positive local type except $(3,3,4)$.

The type $(3,3,4)$ has excess exactly $1$. A second finite computation shows
that if the relative at the other end of its $4$-thread has degree $3$, then
that relative must have type
\[
(1,1,4),\quad (1,2,4),\quad\text{or}\quad(1,3,4),
\]
which provides sufficient negative excess. For higher-degree relatives we
prove a capacity inequality, again reducing only finitely many exceptional
types to the rooted star test. A final one-unit charging rule then gives the
contradiction.

\paragraph{Role of the computer.}
The proof is computer-assisted but not computer-generated. The graph
$K(9,4)$ is constructed directly from its definition, and every
computer-assisted assertion is a finite extension statement explicitly
formulated in the text. The proof-critical certificates are produced by exact C++ bitset programs:
one verifies the rooted-star reductions and a second, independently formulated
program verifies the exceptional $(3,3,4)$ support computation.  These exact
checks use no floating-point arithmetic, random sampling, integer programming,
or heuristic search.  A separate Python/SciPy implementation is retained only
as an independent cross-check and is not needed for the validity of the
computer-assisted lemmas. Appendix~\ref{app:certificates} records the finite
candidate lists and certificates, while Section~\ref{sec:computer} explains
how the verification can be reproduced. The use of a transparent finite
computation is in the same general spirit as many reducibility-and-discharging
arguments; for background on the discharging method see Cranston and
West~\cite{CranstonWest2017}.

\paragraph{Organization.}
Section~\ref{sec:preliminaries} fixes notation and records the geometry of
$K(9,4)$ used later. Section~\ref{sec:minimal-counterexample} introduces a
minimal counterexample and proves the basic reductions, including the rooted
star replacement. Section~\ref{sec:admissible} develops admissible color
sets along threads and the one compatibility computation required for the
exceptional type $(3,3,4)$. Section~\ref{sec:deg3} reduces all positive
local types to this unique exception. Section~\ref{sec:higher-degree}
proves the capacity bounds needed at higher-degree relatives, and
Section~\ref{sec:discharging} completes the proof. Section~\ref{sec:computer}
and the appendices give the exact computational specification and
reproducibility information.

\section{Notation and preliminaries}\label{sec:preliminaries}

\subsection{Graph-theoretic notation}

All graphs considered in this paper are finite and simple. For a graph $G$, we denote by $V(G)$ and $E(G)$ its vertex set and edge set, respectively. For a vertex $v\in V(G)$, its degree is denoted by $\deg_G(v)$, or simply by $\deg(v)$ when the graph is clear from the context.

For a graph $G$, the \emph{average degree} of $G$ is
\[
\operatorname{ad}(G)=\frac{2|E(G)|}{|V(G)|},
\]
and the \emph{maximum average degree} of $G$ is
\[
\operatorname{mad}(G)=\max\{\operatorname{ad}(H): H\subseteq G\}.
\]

We write $\operatorname{odd\mbox{-}girth}(G)$ for the length of a shortest odd cycle in $G$, with the convention that $\operatorname{odd\mbox{-}girth}(G)=\infty$ if $G$ is bipartite.

\medskip

A \emph{walk} in a graph is a sequence $v_0,v_1,\dots,v_\ell$ of vertices such that $v_{i-1}v_i\in E(G)$ for every $i\in\{1,\dots,\ell\}$. Its length is $\ell$. A $u$--$v$ walk is a walk whose first vertex is $u$ and whose last vertex is $v$.

A \emph{thread} in a graph $G$ is a path of length at least $1$ whose internal vertices, if any, all have degree $2$, while both endvertices have degree different from $2$. Thus an edge joining two vertices of degree different from $2$ is a thread of length $1$. If the two endvertices of a thread coincide, we call it a \emph{pinched thread}.

If $u$ and $v$ are the endvertices of a thread, then we say that $u$ and $v$ are \emph{relatives}. For a vertex $v$ of degree at least $3$, we define $\star(v)$ to be the set consisting of $v$ together with all internal vertices of all threads incident with $v$.

Let $v$ be a vertex of degree $k\ge 3$ that does not belong to any pinched thread. Then $v$ belongs to exactly $k$ threads. If the lengths of these threads are
\[
a_1\le a_2\le \dots \le a_k,
\]
then we say that $v$ is an \emph{$(a_1,\dots,a_k)$-vertex}. As usual, a symbol such as $a_i^+$ means that the corresponding thread has length at least $a_i$.

\subsection{The target graph \texorpdfstring{$K(9,4)$}{K(9,4)}}

Our target graph is the Kneser graph $K(9,4)$. We use the convention
\[
V(K(9,4))=\binom{[9]}{4},
\qquad\text{where }[9]=\{0,1,\dots,8\}.
\]
Thus the vertices of $K(9,4)$ are the $4$-subsets of $[9]$, and two such subsets are adjacent if and only if they are disjoint. We refer to the vertices of $K(9,4)$ as \emph{colors}. A homomorphism from a graph $G$ to $K(9,4)$ will be called a \emph{$K(9,4)$-coloring} of $G$.

\begin{lemma}\label{lem:k94-odd-girth}
The odd-girth of $K(9,4)$ is $9$.
\end{lemma}

\begin{proof}
Let $C$ be an odd cycle of length $\ell$ in $K(9,4)$.  For each element
$x\in[9]$, the vertices of $C$ whose corresponding $4$-sets contain $x$
form an independent set of $C$, since adjacent colors are disjoint.  Hence
$x$ occurs in at most $(\ell-1)/2$ colors on $C$.  Counting incidences
between the $\ell$ colors and the nine elements of $[9]$ gives
\[
4\ell\le 9\frac{\ell-1}{2},
\]
and therefore $\ell\ge9$.

Conversely, the following sequence is a $9$-cycle (indices here are ordinary
set elements, not taken modulo $9$):
\[
\begin{aligned}
&0123,\ 4567,\ 1238,\ 0457,\ 1268,\\
&0345,\ 1678,\ 0234,\ 5678,
\end{aligned}
\]
where, for example, $0123$ denotes $\{0,1,2,3\}$.  Consecutive sets, as
well as the last and the first, are disjoint.  Thus the odd-girth is exactly
$9$.
\end{proof}

For later reference it is useful to record the five orbits of ordered pairs
of colors under the natural action of the symmetric group on $[9]$. Fix a
color $X$. The number of colors $Y$ with $|X\cap Y|=r$ is
\[
\binom{4}{r}\binom{5}{4-r},
\]
so the orbit sizes are as follows.

\begin{table}[htbp]
\centering
\caption{Pair orbits in $K(9,4)$ relative to a fixed color $X$.}
\label{tab:pair-orbits}
\begin{tabular}{c|rrrrr}
\toprule
$r=|X\cap Y|$ & 0 & 1 & 2 & 3 & 4\\
\midrule
number of colors $Y$ & 5 & 40 & 60 & 20 & 1\\
\bottomrule
\end{tabular}
\end{table}

This elementary orbit decomposition is one reason why the local computations
remain manageable: any property of a pair of colors that is invariant under
permutations of $[9]$ depends only on $|X\cap Y|$.

If $X,Y\in V(K(9,4))$, then adjacency means $X\cap Y=\emptyset$. Consequently, the distance structure in $K(9,4)$ is governed by the size of the intersection $|X\cap Y|$. In particular:
\[
|X\cap Y|=0 \iff d(X,Y)=1,
\]
\[
|X\cap Y|=4 \iff X=Y \iff d(X,Y)=0.
\]
For the remaining intersection sizes, the possible walk lengths will be encoded through the admissible sets introduced below.

\medskip

Let $X\in V(K(9,4))$ and let $i\ge 1$ be an integer. We define
\[
A_i(X)=\{Y\in V(K(9,4)):\text{ there exists an $X$--$Y$ walk of length }i\text{ in }K(9,4)\},
\]
and
\[
F_i(X)=V(K(9,4))\setminus A_i(X).
\]
Thus $A_i(X)$ is the set of colors that may appear at distance $i$ from $X$ along a thread, while $F_i(X)$ is the set of forbidden colors.

More generally, if $X_1,\dots,X_t\in V(K(9,4))$ and $\ell_1,\dots,\ell_t$ are positive integers, then
\[
\bigcap_{j=1}^t A_{\ell_j}(X_j)
\]
is the set of colors that may be assigned to a vertex joined to the vertices colored $X_1,\dots,X_t$ by pairwise internally disjoint threads of lengths $\ell_1,\dots,\ell_t$, respectively.

The graph $K(9,4)$ is vertex-transitive, so the cardinalities of the sets $A_i(X)$ and $F_i(X)$ do not depend on the particular choice of $X$. In several places we will also use the fact that, up to automorphisms of $K(9,4)$, the relative position of two colors $X,Y$ is determined by the value of $|X\cap Y|$.

\medskip

A substantial part of the proof reduces to finite verification of local extension properties in $K(9,4)$. For this reason, besides the sets $A_i(X)$ and $F_i(X)$, we will also consider intersections of the form
\[
A_{\ell_1}(X_1)\cap \dots \cap A_{\ell_t}(X_t),
\]
and, in the case of two prescribed relatives, we will study their sizes as functions of the intersection size $|X_1\cap X_2|$. All such data will be either proved directly or computed exhaustively and then used as input in the reducibility arguments in later sections.

We conclude this section with a simple observation that will be used repeatedly.

\begin{lemma}\label{lem:long-thread-extension}
Let $G$ be a graph and let $P$ be a thread in $G$ of length $\ell$. Suppose that $\ell$ is large enough so that for every pair of colors $X,Y\in V(K(9,4))$ there exists an $X$--$Y$ walk of length $\ell$ in $K(9,4)$. Then every $K(9,4)$-coloring of $G-V(P)^\circ$ extends to a $K(9,4)$-coloring of $G$, where $V(P)^\circ$ denotes the set of internal vertices of $P$.
\end{lemma}

\begin{proof}
Let $u$ and $v$ be the endvertices of $P$, and let $\varphi$ be a $K(9,4)$-coloring of $G-V(P)^\circ$. By assumption, there exists a $\varphi(u)$--$\varphi(v)$ walk of length $\ell$ in $K(9,4)$. Coloring the consecutive internal vertices of $P$ with the consecutive vertices of this walk yields an extension of $\varphi$ to all of $G$.
\end{proof}

\section{The minimal counterexample and basic reductions}\label{sec:minimal-counterexample}

Suppose, for contradiction, that Theorem~\ref{thm:main} is false.  Among all
counterexamples, choose $G$ lexicographically minimizing
\[
\bigl(n_{\ge3}(G),\ |E(G)|,\ |V(G)|\bigr),
\]
where $n_{\ge3}(G)$ is the number of vertices of degree at least $3$.
This choice is convenient because the replacement graphs used below may
contain more degree-$2$ vertices than $G$, but always contain fewer vertices
of degree at least $3$. The final tie-breaker by $|V(G)|$ is used only for
harmless deletions such as isolated vertices.

\begin{lemma}\label{lem:connected}
The graph $G$ is connected and has minimum degree at least $2$.
\end{lemma}

\begin{proof}
If $G$ were disconnected, a non-$K(9,4)$-colorable component would itself be
a smaller counterexample.  An isolated vertex can be colored arbitrarily,
and a vertex of degree $1$ can be colored after its unique neighbor because
every color of $K(9,4)$ has a neighbor.  Deleting such a vertex and extending
a coloring contradicts minimality.
\end{proof}

\begin{lemma}\label{lem:not-cycle}
The graph $G$ is not a cycle.  Consequently, every degree-$2$ vertex of $G$
belongs to a unique thread whose endvertices have degree at least $3$.
\end{lemma}

\begin{proof}
An even cycle maps to an edge of $K(9,4)$.  If $G=C_\ell$ is an odd cycle,
then $\ell\ge 9$.  Since $K(9,4)$ contains a $9$-cycle, a closed walk of every
odd length at least $9$ is obtained by traversing such a cycle and inserting
backtracking pairs of edges.  Hence $C_\ell$ also maps to $K(9,4)$, a
contradiction.  The second assertion follows from connectedness and the
definition of a thread.
\end{proof}

\begin{lemma}\label{lem:no-long-threads}
The graph $G$ contains no thread of length at least $8$.
\end{lemma}

\begin{proof}
Let $P$ be a thread of length $\ell\ge8$.  Delete the internal vertices of
$P$.  The resulting graph has the same number of vertices of degree at least
$3$ as $G$ and fewer edges, so by minimality it has a $K(9,4)$-coloring.
By Corollary~\ref{cor:all-walks-length-at-least-8}, the colors of the two
endvertices can be joined by a walk of length $\ell$ in $K(9,4)$, and hence
the coloring extends along $P$, a contradiction.
\end{proof}

\begin{lemma}\label{lem:no-pinched-threads}
The graph $G$ contains no pinched thread.
\end{lemma}

\begin{proof}
Let $P$ be a pinched thread of length $\ell$ based at a vertex $v$.  Delete
its internal vertices and color the remaining graph by minimality.  If
$\ell$ is even, every color of $K(9,4)$ is the first and last vertex of a
closed walk of length $\ell$ (repeat an edge back and forth).  If $\ell$ is
odd, then the cycle represented by $P$ and the odd-girth assumption give
$\ell\ge9$, and a closed walk of length $\ell$ exists as in
Lemma~\ref{lem:not-cycle}.  In either case the coloring extends around $P$,
a contradiction.
\end{proof}

We next record the principal replacement used throughout the proof.  The
version below is stated for arbitrary degree; the degree-$3$ instance is the
replacement that motivated the computation in the original draft.

\begin{lemma}[rooted star replacement]\label{lem:star-replacement}
Let $v$ be a vertex of degree $d\ge3$, with incident thread lengths
$a_1,\dots,a_d$, and let $v_1,\dots,v_d$ be the corresponding relatives,
listed with multiplicity if several threads have the same relative.  Fix an
index $r$.  Form $G^*$ by deleting $\star(v)$ and, for every $i\ne r$,
adding an internally new $v_r$--$v_i$ path of length $a_r+a_i$.  If
$v_i=v_r$, this is interpreted as an internally new cycle of that length
based at $v_r$.  Then
\[
\operatorname{odd\mbox{-}girth}(G^*)\ge9,
\qquad
\operatorname{mad}(G^*)<\frac94.
\]
Moreover, $G^*$ has exactly one fewer vertex of degree at least $3$ than
$G$.
\end{lemma}

\begin{proof}
No new internal vertex has degree different from $2$.  For a relative
$x\ne v_r$, let $m_x$ be its multiplicity among the non-root branches.  It
loses exactly $m_x$ incident branch edges and gains exactly $m_x$ incident
new-path edges, so its degree is unchanged.  Suppose that $v_r$ occurs with
total multiplicity $m\ge1$ among $v_1,\dots,v_d$.  The root relative loses
$m$ branch incidences.  Among the new paths it gains one incidence for each
of the $d-m$ branches ending at a different relative and two incidences for
each of the $m-1$ new cycles based at $v_r$.  Its net degree change is
therefore
\[
(d-m)+2(m-1)-m=d-2\ge1.
\]
Thus every relative remains of degree at least $3$, no new vertex of degree
at least $3$ is created, and $v$ disappears.  Hence the number of vertices
of degree at least $3$ decreases by exactly one.

For the odd-girth assertion, let $C$ be an odd cycle of $G^*$.  The new
paths are internally disjoint and have the common root $v_r$, so a simple
cycle uses at most two of them (a new cycle based at $v_r$ can only be used
by itself).  Replace each used new $v_r$--$v_i$ path of length $a_r+a_i$ by
the original $v_r$--$v$ branch of length $a_r$ followed by the original
$v$--$v_i$ branch of length $a_i$.  In the coincident-endpoint case, replace
the new cycle by the corresponding two original branches.  The resulting
object is a closed walk in $G$ of exactly the same length as $C$; repeated
use of the root branch is harmless because we only need a closed walk.
Consequently it is an odd closed walk of length $|C|$.  Every odd closed walk
contains an odd cycle no longer than itself, so an odd cycle of length less
than $9$ in $G^*$ would contradict
$\operatorname{odd\mbox{-}girth}(G)\ge9$.

It remains to check maximum average degree.  Suppose that a subgraph $H$ of
$G^*$ has average degree at least $9/4$.  Repeatedly delete vertices of
degree at most $1$ from $H$; since the current average degree is greater than
$2$, this operation strictly increases the average degree.  We may therefore
assume that $H$ has minimum degree at least $2$.  Consequently, if $H$ uses an edge of a new path, the degree-$2$ internal
vertices force $H$ to contain that path in its entirety (and likewise the
entire newly added cycle in the coincident-endpoint case).  Let $k$ be the
number of new paths used by $H$.

If $k=0$, then $H$ is already a subgraph of $G$.  If $k\ge1$, replace the
$k$ new paths by the corresponding original branches through $v$.  For
$k=1$ this changes neither the number of vertices nor the number of edges.
For $k\ge2$, both numbers decrease by
\[
q=(k-1)a_r,
\]
because the $k$ new paths contain $k$ copies of the root length $a_r$,
whereas the restored star contains the root branch only once.  If $H$ has
$n$ vertices and $m$ edges, then $2m/n\ge9/4>2$, so $m>n$.  Therefore, for
$q>0$,
\[
\frac{2(m-q)}{n-q}>\frac{2m}{n}\ge\frac94.
\]
Thus in all cases we obtain a subgraph of $G$ with average degree at least
$9/4$, contradicting $\operatorname{mad}(G)<9/4$.
\end{proof}

By minimality, every graph $G^*$ from Lemma~\ref{lem:star-replacement} admits
a $K(9,4)$-coloring.  This leads to the following precise finite extension
test.

\begin{lemma}\label{lem:rooted-star-extension}
In the notation of Lemma~\ref{lem:star-replacement}, suppose that for every
choice of colors $X_1,\dots,X_d$ satisfying
\[
X_i\in A_{a_r+a_i}(X_r)\qquad(i\ne r)
\]
we have
\[
A_{a_r}(X_r)\cap\bigcap_{i\ne r}A_{a_i}(X_i)\ne\emptyset.
\]
Then the type $(a_1,\dots,a_d)$ is reducible and cannot occur at $v$ in $G$.
\end{lemma}

\begin{proof}
Color $G^*$ by minimality.  The newly added paths give exactly the displayed
relations between the colors of the relatives.  A color in the displayed
intersection can be assigned to $v$, after which each original incident
thread is colored along a corresponding walk in $K(9,4)$.
\end{proof}

We shall also use a simpler criterion that does not require a replacement.

\begin{lemma}[forbidden-set union bound]\label{lem:global-union-bound}
Let $v$ have incident thread lengths $a_1,\dots,a_d$.  If
\[
\sum_{i=1}^d |F_{a_i}|<126,
\]
then the type $(a_1,\dots,a_d)$ is reducible.
\end{lemma}

\begin{proof}
Delete $\star(v)$ and color the remaining graph by minimality.  If the
relatives receive colors $X_1,\dots,X_d$, then the forbidden colors for $v$
are contained in
\[
\bigcup_{i=1}^d F_{a_i}(X_i).
\]
Its cardinality is strictly less than $126=|V(K(9,4))|$, so at least one
color remains available for $v$, and the coloring extends along all incident
threads.
\end{proof}

\section{Admissible color sets in \texorpdfstring{$K(9,4)$}{K(9,4)}}\label{sec:admissible}

In this section we collect the local facts about $K(9,4)$ that will be used throughout the proof. The main purpose of these definitions and computations is to translate extendability questions in a graph $G$ into finite compatibility questions inside $K(9,4)$.

Recall that
\[
V(K(9,4))=\binom{[9]}{4},
\]
and that two colors are adjacent if and only if they are disjoint.

\subsection{Admissible colors along a thread}

Let $X\in V(K(9,4))$ and let $i\ge 1$ be an integer. Recall that
\[
A_i(X)=\{Y\in V(K(9,4)):\text{ there exists an $X$--$Y$ walk of length }i\text{ in }K(9,4)\},
\]
and
\[
F_i(X)=V(K(9,4))\setminus A_i(X).
\]

Thus, if a vertex $u$ is joined to a vertex colored $X$ by a thread of length $i$, then the set of colors that can be assigned to $u$ is precisely $A_i(X)$, while $F_i(X)$ is the set of forbidden colors.

Since $K(9,4)$ is vertex-transitive, the cardinalities $|A_i(X)|$ and $|F_i(X)|$ do not depend on the choice of $X$. For this reason, when only cardinalities matter, we will simply write $|A_i|$ and $|F_i|$.

\begin{proposition}\label{prop:AiFi-basic}
For every $X\in V(K(9,4))$ and every $i\ge 1$, the set $A_i(X)$ depends only on $i$ and on the intersection pattern with $X$. In particular, the numbers $|A_i(X)|$ and $|F_i(X)|$ are independent of $X$.
\end{proposition}

\begin{proof}
This follows immediately from the vertex-transitivity of $K(9,4)$.
\end{proof}

The exact values of $|A_i|$ and $|F_i|$ for $i\in\{1,\dots,7\}$ are obtained by exhaustive computation.

\begin{proposition}\label{prop:AiFi-values}
For every $X\in V(K(9,4))$, the sets $A_i(X)$ and $F_i(X)$ have the following cardinalities:
\[
|A_1|=5,\qquad |F_1|=121,
\]
\[
|A_2|=21,\qquad |F_2|=105,
\]
\[
|A_3|=45,\qquad |F_3|=81,
\]
\[
|A_4|=81,\qquad |F_4|=45,
\]
\[
|A_5|=105,\qquad |F_5|=21,
\]
\[
|A_6|=121,\qquad |F_6|=5,
\]
\[
|A_7|=125,\qquad |F_7|=1,
\]
\[
|A_8|=126,\qquad |F_8|=0.
\]
Moreover, for every $i\ge 8$ we have
\[
A_i(X)=V(K(9,4))
\qquad\text{and}\qquad
F_i(X)=\emptyset.
\]
\end{proposition}

The same computation can be read at the level of the orbit parameter
$r=|X\cap Y|$. For a fixed $X$, the admissible intersection classes are
listed in Table~\ref{tab:Ai-orbits}.

\begin{table}[htbp]
\centering
\caption{Intersection classes occurring in $A_i(X)$.}
\label{tab:Ai-orbits}
\begin{tabular}{c|l|r|r}
\toprule
$i$ & admissible values of $|X\cap Y|$ & $|A_i|$ & $|F_i|$\\
\midrule
1 & $0$ & 5 & 121\\
2 & $3,4$ & 21 & 105\\
3 & $0,1$ & 45 & 81\\
4 & $2,3,4$ & 81 & 45\\
5 & $0,1,2$ & 105 & 21\\
6 & $1,2,3,4$ & 121 & 5\\
7 & $0,1,2,3$ & 125 & 1\\
8 & $0,1,2,3,4$ & 126 & 0\\
\bottomrule
\end{tabular}
\end{table}

The two tables are mutually consistent: summing the orbit sizes from
Table~\ref{tab:pair-orbits} over the classes in a row of
Table~\ref{tab:Ai-orbits} gives the corresponding value of $|A_i|$.
This provides a compact independent sanity check on the exported data.

\begin{proof}
The values through length $9$ were computed exhaustively; in particular,
$A_8(X)=A_9(X)=V(K(9,4))$ for every $X$.  If a walk of length $i$ joins two
colors, inserting a backtracking pair of edges gives a walk of length $i+2$.
Hence the assertions for lengths $8$ and $9$ imply the statement for every
$i\ge8$.  See Section~\ref{sec:computer}.
\end{proof}

\begin{corollary}\label{cor:all-walks-length-at-least-8}
For every pair of colors $X,Y\in V(K(9,4))$ and every integer $i\ge 8$, there exists an $X$--$Y$ walk of length $i$ in $K(9,4)$.
\end{corollary}

\begin{proof}
Immediate from Proposition~\ref{prop:AiFi-values}.
\end{proof}

\subsection{A two-center compatibility check}

For colors $X,Y\in V(K(9,4))$ and positive integers $a,b$, put
\[
U_{a,b}(X,Y)=A_a(X)\cap A_b(Y).
\]
Thus $U_{a,b}(X,Y)$ is the set of possible colors of a center joined to
$X$ and $Y$ by threads of lengths $a$ and $b$.

The only two-center computation needed in the final proof is the following.

\begin{proposition}\label{prop:334-m4-compatibility}
Let $1\le a\le b\le7$.  Suppose that colors $X_1,X_2,Y_1,Y_2$ satisfy
\[
U_{3,3}(X_1,X_2)\ne\emptyset,
\qquad
U_{a,b}(Y_1,Y_2)\ne\emptyset.
\]
If
\[
U_{a,b}(Y_1,Y_2)\cap
\bigcup_{S\in U_{3,3}(X_1,X_2)} A_4(S)=\emptyset,
\]
then
\[
(a,b)\in\{(1,1),(1,2),(1,3)\}.
\]
Equivalently, outside these three pairs, any two realizable centers of the
indicated types can be joined through the common thread of length $4$.
\end{proposition}

\begin{proof}
This is an exact finite computation in $K(9,4)$.  A direct C++ verifier
enumerates all ordered endpoint pairs, forms all nonempty sets
$U_{3,3}(X_1,X_2)$ and $U_{a,b}(Y_1,Y_2)$, and tests the displayed
intersection.  The only pairs for which an obstruction exists are
$(1,1),(1,2),(1,3)$.  An independent Python implementation based on orbit
states gives the same result.  See Section~\ref{sec:computer}.
\end{proof}

\section{Positive local types}\label{sec:deg3}

\subsection{The excess parameter}
The final contradiction is most transparent after the degree-$2$ vertices are encoded through the lengths of their incident threads.  The quantity introduced below measures exactly how much a vertex contributes above the threshold forced by $\operatorname{mad}(G)<9/4$.

For a vertex $v$ of degree $d\ge3$ with incident thread lengths
$a_1,\dots,a_d$, define
\[
w(v)=\sum_{i=1}^d(a_i-1)
\]
and its \emph{local excess}
\[
\varepsilon(v)=w(v)-(8d-18)
=\sum_{i=1}^d a_i-9d+18.
\]
For $d=3$ this is $a_1+a_2+a_3-9$.

The next lemma is the main finite reduction.  It uses only
Lemma~\ref{lem:global-union-bound}, the rooted star replacement, and the
finite geometry of $K(9,4)$.

\begin{lemma}[positive-type reduction]\label{lem:positive-type-reduction}
Let $v$ be a vertex of $G$ with positive local excess.  Then $v$ has degree
$3$ and type
\[
(3,3,4).
\]
\end{lemma}

\begin{proof}
By Lemma~\ref{lem:no-long-threads}, every incident thread has length at most
$7$.  We therefore enumerate all nondecreasing tuples with entries in
$\{1,\dots,7\}$ and positive local excess.

For degrees $3,4,5,6,7,8$ the numbers of positive tuples are respectively
\[
61,\ 64,\ 41,\ 19,\ 7,\ 2.
\]
The forbidden-set union bound immediately reduces respectively
\[
29,\ 43,\ 38,\ 19,\ 7,\ 2
\]
of them.  For every remaining tuple we apply the rooted star replacement,
using a shortest incident thread as the root, and check the finite extension
condition from Lemma~\ref{lem:rooted-star-extension}.  The only tuple that
fails this extension test is
\[
(3,3,4).
\]
The computation is described in Section~\ref{sec:computer} and reproduced by
\path{anc/scripts/verify_revised_reductions.py}.

For $d\ge9$ no positive tuple exists, since
\[
\varepsilon(v)\le 7d-9d+18=18-2d\le0.
\]
\end{proof}

For clarity, the finite part of the preceding proof is summarized in
Table~\ref{tab:positive-audit}. The column ``union bound'' counts the types
eliminated immediately by Lemma~\ref{lem:global-union-bound}; ``star tested''
counts the remaining types passed to the exact rooted-star verifier.

\begin{table}[htbp]
\centering
\caption{Audit of positive local types.}
\label{tab:positive-audit}
\begin{tabular}{c|r|r|r|l}
\toprule
$d$ & positive types & union bound & star tested & survivors\\
\midrule
3 & 61 & 29 & 32 & $(3,3,4)$\\
4 & 64 & 43 & 21 & none\\
5 & 41 & 38 & 3 & none\\
6 & 19 & 19 & 0 & none\\
7 & 7 & 7 & 0 & none\\
8 & 2 & 2 & 0 & none\\
\bottomrule
\end{tabular}
\end{table}

The table is not an additional hypothesis: it is simply a compact certificate
for the finite enumeration described in Section~\ref{sec:computer}. The
complete source-level candidate audit is included with the ancillary files.

\subsection{The exceptional type \texorpdfstring{$(3,3,4)$}{(3,3,4)}}

Thus the whole positive part of the final discharging argument is reduced to
a single degree-$3$ type.  We next determine what can occur at the other end
of its unique $4$-thread.

\begin{lemma}\label{lem:334-degree3-support}
Let $v$ be a $(3,3,4)$-vertex and let $u$ be the relative of $v$ through its
thread of length $4$.  If $\deg(u)=3$, then $u$ has one of the types
\[
(1,1,4),\qquad(1,2,4),\qquad(1,3,4).
\]
\end{lemma}

\begin{proof}
Write the two other thread lengths at $u$ as $a\le b$.  The two other
threads at $v$ both have length $3$.  Neither can have $u$ as its other
endvertex, because together with the common $4$-thread it would form a
$7$-cycle, contrary to the odd-girth assumption.  Hence all four noncommon
branches below survive when the two centers are deleted.

Let $x_1,x_2$ be the two other relatives of $v$, and let $y_1,y_2$ be the two
other relatives of $u$.  Delete $u,v$, all internal vertices of their
incident threads, and the common $4$-thread.  Add an internally new path of
length $6$ between $x_1$ and $x_2$, and an internally new path of length
$a+b$ between $y_1$ and $y_2$; if the two endpoints of one of these paths
coincide, interpret it as a cycle of the prescribed length.  Call the
resulting graph $H$.

Replacing either new path by the corresponding two original branches through
its deleted center preserves both parity and length.  Hence every odd cycle
of $H$ yields an odd closed walk of the same length in $G$, and therefore
$\operatorname{odd\mbox{-}girth}(H)\ge9$.

Suppose that a subgraph $J$ of $H$ has average degree at least $9/4$.  As in
Lemma~\ref{lem:star-replacement}, delete vertices of degree at most $1$ until
$J$ has minimum degree at least $2$.  Any new path whose edge is used by
$J$ is then used in its entirety.  Replace each such new path by the two
original branches through its deleted center.  A path of length $p+q$ has
$p+q-1$ internal vertices, while the two original branches together with
their center also contribute exactly $p+q-1$ vertices and $p+q$ edges.
Thus this operation preserves both the number of vertices and the number of
edges and produces a subgraph of $G$ with the same average degree, a
contradiction.  Hence $\operatorname{mad}(H)<9/4$.  The degrees of all
external relatives are restored by the replacement paths, while $u$ and $v$
disappear, so $H$ has two fewer vertices of degree at least $3$ than $G$.
By minimality, $H$ admits a $K(9,4)$-coloring.

The length-$6$ replacement path implies
$U_{3,3}(\varphi(x_1),\varphi(x_2))\ne\emptyset$, and the length-$(a+b)$
replacement path implies
$U_{a,b}(\varphi(y_1),\varphi(y_2))\ne\emptyset$: in each case take the
appropriate midpoint of the corresponding walk.  If
$(a,b)\notin\{(1,1),(1,2),(1,3)\}$, Proposition~\ref{prop:334-m4-compatibility}
provides colors for $v$ and $u$ that can be joined by a walk of length $4$.
The coloring then extends independently along all deleted noncommon threads
and along the common $4$-thread, yielding a coloring of $G$, a contradiction.
\end{proof}

\section{Capacity at vertices of degree at least \texorpdfstring{$4$}{4}}\label{sec:higher-degree}

For a type $\tau=(a_1,\dots,a_d)$ let
\[
q_4(\tau)=|\{i:a_i=4\}|
\]
be the number of incident threads of length $4$.  In the final discharging
argument, a vertex can receive charge only through such threads.  The next
finite lemma says that every nonpositive type that would have too little
negative excess to absorb all possible incoming units is itself reducible.

\begin{lemma}[capacity reduction]\label{lem:capacity-reduction}
Let $4\le d\le7$, and let $v$ have a nonpositive type
$\tau=(a_1,\dots,a_d)$.  If
\[
q_4(\tau)>-\varepsilon(\tau),
\]
then $\tau$ is reducible.
\end{lemma}

\begin{proof}
There are finitely many such tuples because every $a_i\le7$.  Their numbers
for degrees $4,5,6,7$ are respectively
\[
11,\ 10,\ 4,\ 1.
\]
For each tuple, the rooted star replacement with a shortest thread as root
satisfies the extension condition of
Lemma~\ref{lem:rooted-star-extension}.  The exhaustive verification is
reported in Section~\ref{sec:computer}; no tuple survives.  The full list of $26$ candidate exceptions, together with their excess and $q_4$ values, is recorded in Appendix~\ref{app:certificates}.
\end{proof}

For large degree no computation is required.

\begin{lemma}\label{lem:capacity-large-degree}
Let $v$ have degree $d\ge8$ and nonpositive excess.  Then
\[
q_4(v)\le-\varepsilon(v).
\]
\end{lemma}

\begin{proof}
Put $q=q_4(v)$.  Since every thread has length at most $7$,
\[
\sum_i a_i\le4q+7(d-q)=7d-3q.
\]
Consequently
\[
-\varepsilon(v)=9d-18-\sum_i a_i
\ge2d-18+3q.
\]
If $q=0$ there is nothing to prove.  If $q\ge1$ and $d\ge8$, then
$d+q\ge9$, and hence
\[
2d-18+3q-q=2(d+q-9)\ge0.
\]
Thus $-\varepsilon(v)\ge q$.
\end{proof}

Combining Lemmas~\ref{lem:positive-type-reduction},
\ref{lem:capacity-reduction}, and~\ref{lem:capacity-large-degree} gives the
form needed later.

\begin{corollary}\label{cor:higher-capacity}
Every vertex $u$ of $G$ with $\deg(u)\ge4$ has $\varepsilon(u)\le0$ and
\[
q_4(u)\le-\varepsilon(u).
\]
\end{corollary}

\begin{proof}
Positive excess is impossible by Lemma~\ref{lem:positive-type-reduction}.
For degrees $4$ through $7$, a violation of the displayed inequality is
reducible by Lemma~\ref{lem:capacity-reduction}; for degree at least $8$ use
Lemma~\ref{lem:capacity-large-degree}.
\end{proof}

\section{The discharging argument}\label{sec:discharging}

Let $n_2$ be the number of degree-$2$ vertices of $G$.  By
Lemma~\ref{lem:not-cycle}, every such vertex belongs to a unique thread, and
therefore
\[
2n_2=\sum_{\deg(v)\ge3} w(v).
\]
On the other hand,
\[
2|E(G)|=2n_2+\sum_{d\ge3}d\,n_d,
\qquad
|V(G)|=n_2+\sum_{d\ge3}n_d,
\]
where $n_d$ denotes the number of vertices of degree $d$.  Since
$\operatorname{mad}(G)<9/4$, in particular $\operatorname{ad}(G)<9/4$, and
hence
\[
n_2>\sum_{d\ge3}(4d-9)n_d.
\]
Multiplying by $2$ and using the thread identity gives
\[
\sum_{\deg(v)\ge3}\varepsilon(v)>0.\tag{1}\label{eq:positive-total-excess}
\]

We now redistribute this excess without changing its total.  By
Lemma~\ref{lem:positive-type-reduction}, the only vertices with positive
initial excess are $(3,3,4)$-vertices, and each of them has excess exactly
$1$.  Every such vertex sends one unit of charge along its unique $4$-thread
to the relative at the other end.  After this transfer its final excess is
$0$.

Consider a vertex $u$ receiving charge.  If $\deg(u)=3$, then
Lemma~\ref{lem:334-degree3-support} gives
\[
\operatorname{type}(u)\in
\{(1,1,4),(1,2,4),(1,3,4)\}.
\]
These three types have initial excess $-3,-2,-1$, respectively, and each has
exactly one incident $4$-thread.  Hence $u$ receives at most one unit and its
final excess remains nonpositive.

If $\deg(u)\ge4$, then the number of units received by $u$ is at most
$q_4(u)$.  By Corollary~\ref{cor:higher-capacity},
\[
q_4(u)\le-\varepsilon(u),
\]
so again the final excess of $u$ is nonpositive.  Vertices receiving no
charge were already nonpositive, except for the $(3,3,4)$-vertices, whose
positive unit has been sent away.

Thus every vertex has nonpositive final excess.  Since the redistribution
preserves the total sum,
\[
\sum_{\deg(v)\ge3}\varepsilon(v)\le0,
\]
contradicting \eqref{eq:positive-total-excess}.  Therefore the minimal
counterexample $G$ cannot exist, and Theorem~\ref{thm:main} follows.

\section{Computer-assisted verification}\label{sec:computer}

This section specifies exactly which assertions are checked by computer and how those checks are performed.  The intent is that a reader should be able to reconstruct the computation from the mathematical description even without reading the implementation.

All finite computations are performed directly in the explicitly defined
$126$-vertex graph $K(9,4)$.  A color is stored as a $9$-bit mask with four
set bits.  Adjacency is disjointness, and $A_i(X)$ is obtained by exact
repeated propagation of the adjacency relation.  No random sampling is used.

\subsection{Exact admissible sets}

The exact verifier obtains
\[
(|F_1|,\dots,|F_8|)=(121,105,81,45,21,5,1,0),
\]
and also checks $A_9(X)=V(K(9,4))$.  Together with insertion of backtracking
pairs, this proves Proposition~\ref{prop:AiFi-values} for all lengths used in
the proof.

\subsection{Exact rooted-star test}

Fix a rooted type $(a_1,\dots,a_d)$ with root length $a_1$ and, by
vertex-transitivity, fix the root-relative color $X$.  The possible colors of
the deleted center through the root branch form
\[
C=A_{a_1}(X).
\]
For a non-root branch of length $a_i$, the replacement path requires its
endpoint color $Y_i$ to lie in $A_{a_1+a_i}(X)$.  This endpoint forbids from
$C$ precisely
\[
F_{a_i}(Y_i)\cap C.
\]
Hence the rooted star fails to extend exactly when one may choose one allowed
endpoint on every non-root branch so that these forbidden subsets cover
$C$.

The primary verifier, \path{anc/cpp/verify_revised_reductions.cpp}, solves this
as an exact finite set-cover problem using two $64$-bit words for subsets of
the $126$ colors.  Dominated masks are discarded and the remaining choices
are checked exhaustively by meet-in-the-middle.  Thus the result is purely
combinatorial and does not depend on numerical optimization tolerances.

For positive local excess, the exact output is
\[
\begin{array}{c|r|r|r|l}
d&\text{positive}&\text{union bound}&\text{star tested}&\text{survivors}\\
\hline
3&61&29&32&(3,3,4)\\
4&64&43&21&\emptyset\\
5&41&38&3&\emptyset\\
6&19&19&0&\emptyset\\
7&7&7&0&\emptyset\\
8&2&2&0&\emptyset
\end{array}
\]
and this proves Lemma~\ref{lem:positive-type-reduction}.

For the capacity test of Lemma~\ref{lem:capacity-reduction}, the numbers of
candidate exceptions in degrees $4,5,6,7$ are respectively
\[
11,\ 10,\ 4,\ 1,
\]
and every candidate is reducible by the rooted-star test.

\subsection{Exact support test for \texorpdfstring{$(3,3,4)$}{(3,3,4)}}

The independent program \path{anc/cpp/verify_334_support.cpp} does not use the
state abstraction.  It enumerates all ordered endpoint pairs and all
nonempty sets $U_{3,3}(X_1,X_2)$ and $U_{a,b}(Y_1,Y_2)$ directly.  For each
$1\le a\le b\le7$ it tests whether an incompatible pair exists through a
thread of length $4$.  The output is exactly
\[
(1,1),\quad(1,2),\quad(1,3),
\]
proving Proposition~\ref{prop:334-m4-compatibility}.  The Python implementation
\path{anc/scripts/verify_revised_reductions.py} independently obtains the same
three pairs by grouping endpoint pairs into orbit states.

\subsection{Verification architecture}

The proof-relevant software is deliberately small. The module
\path{anc/src/k9_4/kneser94.py} generates the $126$ colors and the disjointness
adjacency relation. The module \path{anc/src/k9_4/admissible.py} computes exact
walk-reachability sets $A_i(X)$ and their complements. The Python program
\path{anc/scripts/verify_revised_reductions.py} performs a readable independent
cross-check of the positive-type, capacity, and support audits.  Its rooted-star
subroutine is formulated as a mixed-integer feasibility problem and solved by
SciPy/HiGHS; it is therefore not used as the proof certificate.  The exact C++
verifiers described below are authoritative for the computer-assisted lemmas.

The independent C++ program \path{anc/cpp/verify_revised_reductions.cpp} stores
a subset of the $126$ colors in two $64$-bit words. For a rooted type it
fixes the root-relative color by vertex-transitivity, constructs the exact
family of forbidden masks contributed by every possible non-root endpoint,
and asks whether one mask from each branch can cover all root-admissible
colors. Dominated masks are removed, after which the remaining finite
set-cover instance is checked exhaustively. The second C++ program,
\path{anc/cpp/verify_334_support.cpp}, deliberately uses a different formulation:
it enumerates endpoint pairs and the sets $U_{a,b}$ directly, without the
state abstraction used in the Python code.

In pseudocode, the rooted-star test is:
\begin{enumerate}[label=\arabic*.]
  \item Fix a root color $X$ and compute $C=A_{a_r}(X)$.
  \item For each non-root branch $i$, enumerate all endpoint colors
        $Y\in A_{a_r+a_i}(X)$.
  \item Replace each endpoint $Y$ by the forbidden mask
        $C\cap F_{a_i}(Y)$.
  \item Ask whether one mask from each non-root branch can cover $C$.
  \item If no such cover exists, every coloring of the replacement graph
        extends across the deleted star.
\end{enumerate}
All sets in this procedure are exact finite sets; there are no numerical
tolerances.

\subsection{Scope of the computer assistance}

Only four kinds of statements in the proof depend on exhaustive computation:
\begin{enumerate}[label=(C\arabic*)]
  \item the values of $A_i(X)$ and $F_i(X)$ for the short thread lengths;
  \item the rooted-star extension test used in
        Lemma~\ref{lem:positive-type-reduction};
  \item the rooted-star capacity test used in
        Lemma~\ref{lem:capacity-reduction};
  \item the direct compatibility test in
        Proposition~\ref{prop:334-m4-compatibility}.
\end{enumerate}
The graph-theoretic reductions, preservation of odd girth and maximum
average degree, the excess identity, the large-degree capacity inequality,
and the final charging argument are entirely handwritten.

\subsection{Reproducibility}

\paragraph{Code availability.}
The proof-relevant source code, exact verification programs, reference outputs,
and reproducibility material are archived on Zenodo as version 1.0.0
\cite{NazarczukCode2026}, DOI
\href{https://doi.org/10.5281/zenodo.22899801}{\nolinkurl{10.5281/zenodo.22899801}}.
The development repository is available on
\href{https://github.com/mar2000/chen-raspaud-k4-computer-verification}{GitHub}.
The \path{anc/} directory included with this arXiv submission contains the
proof-relevant code and reference outputs, so all computer-assisted claims can
also be checked without accessing an external service.

The exact checks can be reproduced with
\begin{verbatim}
g++ -std=c++20 -O2 anc/cpp/verify_revised_reductions.cpp -o verify_star
./verify_star

g++ -std=c++20 -O2 anc/cpp/verify_334_support.cpp -o verify_334
./verify_334
\end{verbatim}
The expected terminal lines are
\begin{verbatim}
ALL EXACT C++ STAR CHECKS PASSED
ALL EXACT C++ 334 SUPPORT CHECKS PASSED
\end{verbatim}
As a non-authoritative independent cross-check, the Python verification is run with
\begin{verbatim}
PYTHONPATH=anc/src python anc/scripts/verify_revised_reductions.py
\end{verbatim}
and ends with
\begin{verbatim}
ALL REVISED REDUCTION CHECKS PASSED
\end{verbatim}

\section{Concluding remarks}

The proof isolates a relatively small amount of computation from the global
argument. Once the exact local extension tests are available, the final
structure is simple: every positive type is reducible except $(3,3,4)$, and
the only remaining positive unit can be absorbed through its unique
$4$-thread. In particular, the computer is not used to search for a charging
rule; it certifies the finite local statements that make the one-unit rule
valid.

The rooted-star formulation may also be useful independently of the present
parameter. It converts a local extension problem into a finite covering
question inside the target graph, separating graph-theoretic preservation
lemmas from target-specific computation. Whether this point of view can be
made sufficiently uniform to be useful for other Kneser targets is a natural
methodological question, but no such generalization is needed for
Theorem~\ref{thm:main}.

\appendix

\section{Finite certificates}\label{app:certificates}

This appendix records the finite data used by the two nontrivial computer
checks. The tables are generated from the same outputs as the accompanying
verifiers and are included so that the size and shape of the exhaustive
search are visible directly in the manuscript.

\subsection{Capacity exceptions}

For $4\le d\le7$, Lemma~\ref{lem:capacity-reduction} needs to check only
nonpositive types with $q_4>-\varepsilon$. There are exactly $26$ such
types. Every row below is certified reducible by the rooted-star verifier.

\begin{longtable}{c|l|r|r}
\caption{All finite capacity exceptions.}\label{tab:capacity-exceptions}\\
\toprule
$d$ & type $\tau$ & $\varepsilon(\tau)$ & $q_4(\tau)$\\
\midrule
\endfirsthead
\toprule
$d$ & type $\tau$ & $\varepsilon(\tau)$ & $q_4(\tau)$\\
\midrule
\endhead
4 & $(1,4,6,7)$ & 0 & 1\\
4 & $(2,4,4,7)$ & $-1$ & 2\\
4 & $(2,4,5,7)$ & 0 & 1\\
4 & $(2,4,6,6)$ & 0 & 1\\
4 & $(3,4,4,6)$ & $-1$ & 2\\
4 & $(3,4,4,7)$ & 0 & 2\\
4 & $(3,4,5,6)$ & 0 & 1\\
4 & $(4,4,4,4)$ & $-2$ & 4\\
4 & $(4,4,4,5)$ & $-1$ & 3\\
4 & $(4,4,4,6)$ & 0 & 3\\
4 & $(4,4,5,5)$ & 0 & 2\\
\midrule
5 & $(2,4,7,7,7)$ & 0 & 1\\
5 & $(3,4,6,7,7)$ & 0 & 1\\
5 & $(4,4,4,6,7)$ & $-2$ & 3\\
5 & $(4,4,4,7,7)$ & $-1$ & 3\\
5 & $(4,4,5,6,7)$ & $-1$ & 2\\
5 & $(4,4,5,7,7)$ & 0 & 2\\
5 & $(4,4,6,6,6)$ & $-1$ & 2\\
5 & $(4,4,6,6,7)$ & 0 & 2\\
5 & $(4,5,5,6,7)$ & 0 & 1\\
5 & $(4,5,6,6,6)$ & 0 & 1\\
\midrule
6 & $(4,4,6,7,7,7)$ & $-1$ & 2\\
6 & $(4,4,7,7,7,7)$ & 0 & 2\\
6 & $(4,5,6,7,7,7)$ & 0 & 1\\
6 & $(4,6,6,6,7,7)$ & 0 & 1\\
\midrule
7 & $(4,6,7,7,7,7,7)$ & 0 & 1\\
\bottomrule
\end{longtable}

\subsection{Direct support audit for \texorpdfstring{$(3,3,4)$}{(3,3,4)}}

The program \path{anc/cpp/verify_334_support.cpp} enumerates every pair
$1\le a\le b\le7$. The middle column below is the number of distinct
nonempty sets $U_{a,b}(Y_1,Y_2)$ encountered over all ordered endpoint
pairs. The last column records whether there exists a realization
incompatible with a realizable $U_{3,3}$ across a $4$-thread.

\begin{longtable}{c|r|c}
\caption{Direct compatibility certificate through a $4$-thread.}\label{tab:334-certificate}\\
\toprule
$(a,b)$ & distinct nonempty $U_{a,b}$ sets & incompatible realization?\\
\midrule
\endfirsthead
\toprule
$(a,b)$ & distinct nonempty $U_{a,b}$ sets & incompatible realization?\\
\midrule
\endhead
$(1,1)$ & 252 & yes\\
$(1,2)$ & 1386 & yes\\
$(1,3)$ & 1386 & yes\\
$(1,4)$ & 1386 & no\\
$(1,5)$ & 1386 & no\\
$(1,6)$ & 756 & no\\
$(1,7)$ & 756 & no\\
$(2,2)$ & 5166 & no\\
$(2,3)$ & 12726 & no\\
$(2,4)$ & 12726 & no\\
$(2,5)$ & 10206 & no\\
$(2,6)$ & 5796 & no\\
$(2,7)$ & 2772 & no\\
$(3,3)$ & 7686 & no\\
$(3,4)$ & 15246 & no\\
$(3,5)$ & 13356 & no\\
$(3,6)$ & 10332 & no\\
$(3,7)$ & 5796 & no\\
$(4,4)$ & 8001 & no\\
$(4,5)$ & 15372 & no\\
$(4,6)$ & 13356 & no\\
$(4,7)$ & 10332 & no\\
$(5,5)$ & 8001 & no\\
$(5,6)$ & 15372 & no\\
$(5,7)$ & 13356 & no\\
$(6,6)$ & 8001 & no\\
$(6,7)$ & 15372 & no\\
$(7,7)$ & 8001 & no\\
\bottomrule
\end{longtable}

Thus the only incompatible shorter-length pairs are exactly
$(1,1),(1,2),(1,3)$, as asserted in
Proposition~\ref{prop:334-m4-compatibility}.

\section{Source-code map and reproducibility}\label{app:code}

The following files are sufficient to reproduce every computer-assisted
claim used in the final proof.

\begin{longtable}{p{0.34\textwidth}|p{0.58\textwidth}}
\caption{Proof-relevant source files.}\label{tab:source-map}\\
\toprule
file & purpose\\
\midrule
\endfirsthead
\toprule
file & purpose\\
\midrule
\endhead
\path{anc/src/k9_4/kneser94.py} & construction of the $126$ colors, disjointness adjacency, and basic bitmask helpers\\
\path{anc/src/k9_4/admissible.py} & exact walk-reachability sets $A_i(X)$ and forbidden sets $F_i(X)$\\
\path{anc/scripts/verify_revised_reductions.py} & readable Python/SciPy cross-check of the positive-type audit, capacity audit, and degree-$3$ support pairs; not used as the exact proof certificate\\
\path{anc/cpp/verify_revised_reductions.cpp} & independent exact bitset verifier for admissible sets, positive types, and capacity exceptions\\
\path{anc/cpp/verify_334_support.cpp} & independent direct enumeration for the $(3,3,4)$ support proposition\\
\path{anc/run_all_checks.sh} & convenience script compiling and running the exact verifiers\\
\bottomrule
\end{longtable}

A minimal reproduction consists of a C++20 compiler and Python~3. The C++
checks are run with
\begin{verbatim}
g++ -std=c++20 -O2 anc/cpp/verify_revised_reductions.cpp -o verify_star
./verify_star

g++ -std=c++20 -O2 anc/cpp/verify_334_support.cpp -o verify_334
./verify_334
\end{verbatim}
and the expected final lines are
\begin{verbatim}
ALL EXACT C++ STAR CHECKS PASSED
ALL EXACT C++ 334 SUPPORT CHECKS PASSED
\end{verbatim}
The independent Python cross-check is
\begin{verbatim}
PYTHONPATH=anc/src python anc/scripts/verify_revised_reductions.py
\end{verbatim}
and terminates with
\begin{verbatim}
ALL REVISED REDUCTION CHECKS PASSED
\end{verbatim}
The raw outputs supplied with the manuscript contain the complete candidate
counts and certificate tables appearing above.

\section*{Chronology and independent development}
The mathematical research leading to the present proof was carried out through
the end of 2025. The proof was written up in April 2026. While the manuscript
was being prepared for public release, Choi independently published a proof of
the same $k=4$ case~\cite{Choi2026}; the article was first published online on
29 July 2026. As noted in the Introduction, the two proofs share the broad
minimal-counterexample, thread-reduction, and discharging framework, but the
local arguments are different: Choi's proof is computer-free, while the present
proof uses the finite geometry of $K(9,4)$, the rooted-star replacement, and
exact exhaustive verification of the remaining local extension statements. The present proof was developed independently and does not rely on any
result or argument from Choi's paper. The citation is included to record the
chronology and to place the present independently developed argument in the
context of the subsequently published result.

\end{document}